\documentclass[12pt]{article} 

\usepackage{times}
\usepackage{latexsym}
\usepackage{amsmath}
\usepackage{amsfonts}
\usepackage{xcolor}
\usepackage{tikz}
\usepackage{amsthm}
\usepackage{hyperref}
\usepackage{parskip}
\usepackage{enumitem}

\begingroup
    \makeatletter
    \@for\theoremstyle:=definition,remark,plain\do{%
        \expandafter\g@addto@macro\csname th@\theoremstyle\endcsname{%
            \addtolength\thm@preskip\parskip
            }%
        }
\endgroup
\makeatletter
\newtheorem*{rep@theorem}{\rep@title}
\newcommand{\newreptheorem}[2]{%
\newenvironment{rep#1}[1]{%
 \def\rep@title{#2 \ref{##1}}%
 \begin{rep@theorem}}%
 {\end{rep@theorem}}}
\makeatother

\usepackage{epstopdf}
\newtheorem{Definition}[equation]{Definition}
\newcounter{mark}

\newtheorem{Proposition}[equation]{Proposition}
\newtheorem{Lemma}[equation]{Lemma}

\newtheorem{Remark}[equation]{Remark}
\newtheorem{Problem}[equation]{Problem}

\newtheorem{Question}[equation]{Question}
\newtheorem{Coupling}[mark]{Coupling}

\theoremstyle{definition}
\newtheorem{Example}[equation]{Example}

\definecolor{darkred}{rgb}{0.7,0,0} 
\newcommand{\darkred}{\color{darkred}} 
\newcommand{\defn}[1]{\emph{\darkred #1}} 
\newcommand{\un}{\pi}

\newreptheorem{Theorem}{Theorem}
\newreptheorem{Mark}{Marking Scheme}

\newcommand{\E}[1]{\mathbb{E}\!\left(#1\right)}

\newcommand{\eps}{\epsilon}
\newcommand{\f}[2]{\frac{#1}{#2}}
\renewcommand{\P}{\mathcal{P}}
\newcommand{\G}{\mathcal{G}}
\DeclareMathOperator{\Var}{Var}

\newcommand{\tme}{t_{\text{mix}}(\epsilon)}

\newcommand{\M}{\mathcal{M}}

\newcommand{\figtopfourhitting}{\begin{tikzpicture}[->,auto,node distance=2.5cm, main node/.style={circle,draw,font=\Large\bfseries}]

  \node[main node] (0) {0};
  \node[main node] (1) [right of=0] {1};
	\node[main node] (2) [right of=1] {2};
	\node[main node] (3) [right of=2] {3};
	\node[main node] (4) [right of=3] {4};

  \path
    (0) edge node {$\frac{n}{n}$} (1)
    (1) edge [loop below] node {$\frac{1}{n}$} (1)
        edge node {$\frac{n-1}{n}$} (2)
    (2) edge [loop below] node {$\frac{2}{n}$} (2)
        edge node {$\frac{n-2}{n}$} (3)
    (3) edge [loop below] node {$\frac{3}{n}$} (3)
        edge node {$\frac{n-3}{n}$} (4);
\end{tikzpicture}
}

\newcommand{\figafteronehitting}{\begin{tikzpicture}[->,auto,node distance=2.5cm,main node/.style={circle,draw,font=\Large\bfseries}]

  \node[main node] (10) {1,0};
	\node[main node] (20) [right of=10] {2,0};
	\node[main node] (30) [right of=20] {3,0};
	\node[main node,fill=blue!30] (40) [right of=30] {4,0};
	\node[main node,fill=blue!30] (21) [below of=20] {2,1};
	\node[main node,fill=blue!30] (31) [below of=30] {3,1};
	\node[main node,fill=blue!30] (41) [below of=40] {4,1};
	\node[main node,fill=blue!30] (32) [below of=31] {3,2};
	\node[main node,fill=blue!30] (42) [below of=41] {4,2};
	\node[main node,fill=blue!30] (43) [below of=42] {4,3};
	\node[main node, node distance=2.5cm*1.41] (3) [below left of=43] {3};
	\node[main node] (2) [left of=3] {2};
	\node[main node] (1) [left of=2] {1};
	\node[main node] (0) [left of=1] {0};

	\path[every node/.style = near start]
    (0) edge node {$\frac{3}{4}$} (1) 
				edge node {$\frac{1}{4}$} (10)
		(1) edge node {$\frac{2}{4}$} (2)
		    edge node {$\frac{1}{4}$} (21)
		(2) edge node {$\frac{1}{4}$} (3)
		    edge node {$\frac{1}{4}$} (32)
		(3) edge node {$\frac{1}{4}$} (43)
		(10) edge node {$\frac{3}{4}$} (20)
		(20) edge node {$\frac{2}{4}$} (30)
				 edge[bend left] node {$\frac{1}{4}$} (21)
		(21) edge node {$\frac{2}{4}$} (31)
				 edge node {$\frac{1}{4}$} (20)
		(30) edge node {$\frac{1}{4}$} (40)
				 edge[bend left] node {$\frac{1}{4}$} (32)
		(31) edge node {$\frac{1}{4}$} (41)
				 edge node {$\frac{1}{4}$} (30)
				 edge[bend left] node {$\frac{1}{4}$} (32)
		(32) edge node {$\frac{1}{4}$} (42)
				 edge node {$\frac{2}{4}$} (31)
		(40) edge[bend left] node {$\frac{1}{4}$} (43)
		(41) edge[bend left] node {$\frac{1}{4}$} (43)
				 edge node {$\frac{1}{4}$} (40)
		(42) edge[bend left] node {$\frac{1}{4}$} (43)
				 edge node {$\frac{2}{4}$} (41)
		(43) edge node {$\frac{3}{4}$} (42);
\end{tikzpicture}  }

\title{Coupling for features of random walks}
\author{Graham White}
\date{\today}
\begin{document}

\maketitle

\begin{abstract}
We use coupling to study the time taken until the distribution of a statistic on a Markov chain is close to its stationary distribution. Coupling is a common technique used to obtain upper bounds on mixing times of Markov chains, and we explore how this technique may be used to obtain bounds on the mixing of a statistic instead.
\end{abstract}

\section{Introduction}
\label{sec:intro}

We are interested in the following general problem.

\begin{Problem}
If $\M$ is a Markov chain, and $f$ is a function defined on the states of $\M$, how long must $\M$ be run to guarantee that the distribution of $f$ is close to what it would be on the stationary distribution of $\M$?
\end{Problem}

Much is known about various schemes for shuffling a deck of cards (\cite{dovetail}, \cite{randomtranspositions}), and how many shuffles are necessary before the deck is `random'. In some circumstances it might not be necessary that the entire deck be random, but just some part of it. For example, perhaps a certain game of poker only uses the top $17$ cards of the deck. In playing this game, only the identity and order of these top $17$ cards are important, not the order of the entire deck. It might be expected that to randomise the cards in these positions, fewer shuffles are required than are necessary to randomise the entire deck. This is one instance of the problem --- given a shuffling scheme, how many iterations are required to randomise the top $17$ cards of the deck? 

The same question can be asked for other choices of $f$ --- how long until the four bridge hands dealt in blocks from this deck are random? That is, the sets of cards in positions 1-13, 14-26, 27-39 and 40-52, but not their exact locations within these blocks. What if the dealing is done to each player in turn rather than in blocks of $13$ consecutive cards? How many shuffles are necessary to randomise the location of the ace of spades, or the identity of the card immediately following the ace of spades, or the distance between the aces of spades and hearts?

Problems of this sort have been considered previously. \cite{ADS} studies the mixing time of a deck of cards where certain sets of cards are identified, for instance where the suits of cards do not matter, or all face cards are equivalent. That paper gives an explicit formula for the separation distance after $t$ steps in such a setting. \cite{better} discusses this same problem as well as the dual problem of asking about the hands of cards dealt from a shuffled deck, ignoring the order in which those cards were dealt. This may be seen as identifying sets of positions rather than sets of cards. \cite{better} also presents intriguing computational data showing that the number of riffle shuffles required for this latter problem changes depending on the dealing method used, that is, that identifying different sets of positions produces different results. In various contexts, the values of the function $f$ might be referred to as `statistics' or `features' of the Markov chain.

Both \cite{JCthesis} and \cite{ADS} show that the mixing time for the position of a single card under riffle shuffles is $\log_2(n) + c$, \cite{JCthesis} by calculating the eigenvalues of the walk and also via a coupling argument, and \cite{ADS} by explicit calculations. \cite{repeatedcards1} simulates the required number of shuffles for the games of bridge and blackjack. The introduction of \cite{ADS} has additional references and background.

Previous work on these problems has mostly involved explicit calculations and formulas, which are then analysed with calculus. An exception is \cite{rabinovich2016function}, which develops the use of eigenfunctions, where the statistic in question is expanded in an eigenbasis of the original chain. Appendix B of \cite{ADS} also contains a computation in this style.  

The contributions of this paper and the sequel \cite{GWsstfeatures} are technical --- we describe how coupling (and in the sequel, strong stationary times) may be adapted to give upper bounds on the mixing of a function on a Markov chain. These appear to be the first general approaches which use more probabilistic methods. 

In general, the answers to these problems will depend on the function $f$. There are choices for $f$ where the distribution of $f$ is correct after one step, and there are choices where the distribution of $f$ is not correct until the whole chain is near its stationary distribution. Instances of each behaviour are shown in Section \ref{sec:examples}. We will also examine what known couplings (those already used for upper bounds on mixing times) have to say about various choices of $f$. 

In some cases the statistic $f$ may form a Markov chain in its own right, as a quotient chain of $\M$. For example, when shuffling a deck of cards, the location of the ace of spades at time $(t+1)$ depends only on its location at time $t$. In contrast, the knowledge of which cards are in the top half of the deck at time $t$ is usually not enough information to determine which cards will be in the top half of the deck at time $(t+1)$. We do not require that $f$ be a quotient chain of $\M$. When this does occur, the analysis will not use this fact, because the goal is to demonstrate techniques that are applicable in more general settings. For this reason, some of the simpler examples may work more nicely than expected.

In this paper, upper bounds on mixing times will come from coupling arguments, so mixing times will be according to total variation distance.

\begin{Definition}
Let $\M$ be an aperiodic and irreducible Markov chain and $f$ be a function on the state space of $\M$. The \defn{stationary distribution of $f$} is the distribution of $f$ on the stationary distribution of $\M$. 
\end{Definition}

\begin{Example}
\label{ex:statlist}
Our first examples are shuffling schemes on a deck of $n$ cards. That is, they will be random walks on the group $S_n$.

Here are some statistics of interest. For riffle shuffles, statistics involving the locations of certain cards or cards in particular locations have been analysed in \cite{ADS} and \cite{repeatedcards1} in more detail than in this paper. The focus here is on developing probabilistic techniques for these problems. 
\begin{itemize}
\item The value of the top, second-to-top, bottom, or $k$th card.
\item The values and order of the top $k$ cards, or of the cards in a particular set of positions.
\item The set of cards in a particular set of positions, ignoring their relative order. For example, the sets of cards in the top quarter of the deck, the next quarter, the next quarter, and the bottom quarter, as might be relevant if one were to deal cards in blocks. Alternatively, the sets of cards in positions congruent to $i$ modulo four, for each $i$, as if cards were to be dealt one at a time.
\item The location of a particular card or set of cards.
\item The parity of the permutation.  
\end{itemize}
\end{Example}

Some answers to questions like these are the following, which are proven in Section \ref{sec:examples}. 

\begin{Proposition}
Using the random-to-top shuffle on a deck of $n$ cards, it takes $n\log(n)$ steps to get the entire deck close to random (via a standard coupon-collector argument), but only $n\log(\frac{n}{n-16})$ steps to get the top 17 cards close to random, or $n\log(\frac{3}{2})$ steps to get the top third of the deck close to random.  
\end{Proposition}

\begin{Proposition}
Using inverse GSR riffle shuffles on a deck of $n$ cards, it takes $\frac{3}{2}\log_2(n)$ steps to get the entire deck close to random, but only $\log_2(n)$ steps to get any of the following statistics close to random: The identity of the top card, the location of the ace of spades, the set of cards in the top quarter of the deck, or the sets of cards in each quarter of the deck.
\end{Proposition}

\begin{Example}
\label{ex:hypercube}
Consider the random walk on the $n$--dimensional hypercube where at each step there is a $\frac{1}{2}$ chance to move to a random neighbour and a $\frac{1}{2}$ chance to remain still. This can also be considered as a random walk on $n$--bit binary strings, where at each step a random bit is chosen and replaced with $0$ or $1$ with equal probability.

How long does it take until statistics such as the following are close to their stationary distributions?
\begin{itemize}
\item The value of the first bit
\item The number of `1' bits
\item The location of the first `1'.
\end{itemize}
\end{Example}

Section \ref{sec:hypercube} shows that the value of the first bit mixes after $n$ steps, and that the location of the first 1 mixes after $O(n)$ steps. This should be compared to the mixing time of the walk, which is $\frac{1}{2}n\log(n)$.


\subsection*{Acknowledgements}

I am grateful to my advisor, Persi Diaconis, for introducing me to this field and for many helpful discussions.

\section{Coupling for features of random walks}
\label{sec:coupling}

This section gives results relating coupling to the convergence of a statistic on a Markov chain. These will be used in Section \ref{sec:examples} to give examples of bounds on the convergence of the statistics mentioned in Example \ref{ex:statlist} for some simple shuffling techniques --- random-to-top shuffles, inverse riffle shuffles, and random transpositions.

For the sake of comparison, we will first give a proof that couplings give bounds on the mixing time of the whole Markov chain. This is a classical result (Theorem 5.2 of \cite{LPW}), but the perspective on the problem will be useful for the material which follows. 

\begin{Proposition}
\label{prop:coupling1}
Let $C$ be a coupling on two instances of a Markov chain $\M$, $p$ be between 0 and 1, and $t$ be a positive integer. Let $X_t$ and $Y_t$ be the states of the two instances of $\M$ after taking $t$ steps according to $C$. If for any initial states $x_0$ and $y_0$, there is at least a probability $p$ that $X_t = Y_t$, then for any initial state $x_0$ the distribution of $X_t$ is within $(1-p)$ of the stationary distribution of $\M$ in total variation distance.
\end{Proposition} 
\begin{proof}
A condition of this result is that for any initial states $x_0$ and $y_0$, there is at least a probability $p$ that $X_t = Y_t$. From Lemma \ref{lem:linearcombination}, this is also true if $X_0$ and $Y_0$ are allowed to be distributions. This result is used with $X_0$ being an arbitrary fixed state and $Y_0$ being the stationary distribution $\un $. 

Consider $X_t$ and $Y_t = \un M^t = \un$. To show that these distributions overlap in most of their area, let $\P_1$ be the set of paths of length $t$ starting at $X_0$, and let $\P_2$ be the set of paths of length $t$ starting at $\un $, all paths being weighted by their probabilities. The goal is to pair proportion $p$ of the paths from $\P_1$ with paths from $\P_2$ which end at the same place. Because $X_t$ is the distribution of endpoints of paths in $\P_1$ and $Y_t$ is the distribution of endpoints of paths in $\P_2$, this will guarantee that these two distributions overlap in $p$ of their area. 

This pairing is given by the coupling $C$. Start with two copies of $\M$, one in the state $X_0$ and the other in the stationary distribution $\un $, and evolve them according to the coupling $C$. Pair the paths taken by the two chains. This is a pairing between paths from $\P_1$ and paths from $\P_2$ because $C$ is a coupling, so the behaviour in either chain is what it would be in isolation. There is at least probability $p$ that the two chains end in the same state, so at least $p$ of the paths in $\P_1$ are paired with a path from $\P_2$ which ends at the same state. Therefore $X_t$ and $Y_t = \un$ overlap in at least $p$ of their area, as required.
\end{proof}

\begin{Lemma}
\label{lem:linearcombination}
Using the notation of Proposition \ref{prop:coupling1}, assume that for any two initial states $x_0$ and $y_0$, there is at least a probability $p$ that $X_t = Y_t$ when the two chains evolve according to the coupling $C$. Then if the two chains are started in arbitrary distributions $X_0$ and $Y_0$ rather than fixed states, there is still at least probability $p$ that $X_t = Y_t$.
\end{Lemma}
\begin{proof}
Conditioned on any pair of initial states $x_0$ and $y_0$, the probability that $X_t = Y_t$ is at least $p$. Averaging these probabilities over the distributions $X_0$ and $Y_0$ gives the required result.
\end{proof}

We now adapt Proposition \ref{prop:coupling1} to apply to features of a Markov chain.

\begin{Proposition}
\label{prop:coupling2}
As in Proposition \ref{prop:coupling1}, let $C$ be a coupling on two copies of a Markov chain $\M$, $p$ be between $0$ and $1$, $t$ a positive integer, and $f$ a function on the state space $\Omega$. If for any initial states $x_0$ and $y_0$ there is at least a probability $p$ that $f(X_t) = f(Y_t)$ when the chains $X$ and $Y$ evolve according to the coupling $C$, then for any initial state $x_0$ the distribution $f(X^t)$ is within $(1-p)$ of $f(\un )$, the stationary distribution of $f$, in total variation distance.
\end{Proposition} 
\begin{proof}
This proof is very similar to the proof of Proposition \ref{prop:coupling1}. As with that proof, start by using Lemma \ref{lem:linearcombinationf} to show that for any initial distributions $X_0$ and $Y_0$, there is at least a probability $p$ that $f(X_t) = f(Y_t)$.

Let $X_0$ be any initial distribution, and $Y_0$ be the stationary distribution $\un$. Consider $f(X_t)$ and $f(Y_t) = f(\un M^t) = f(\un)$. As in the proof of Proposition \ref{prop:coupling1}, let $\P_1$ be the set of paths of length $t$ starting at $X_0$, and let $\P_2$ be the set of paths of length $t$ starting at $\un $, weighted by their probabilities. The goal is to pair proportion $p$ of the paths from $\P_1$ with paths from $\P_2$ which end, not necessarily at the same state, but at a state with the same value of $f$. This guarantees that the two distributions $f(X_t)$ and $f(Y_t) = f(\un)$ overlap in $p$ of their area. 

As in the proof of Proposition \ref{prop:coupling1}, the pairing is given by the coupling $C$. Start with two copies of $\M$, one in the distribution $X_0$ and the other in the stationary distribution $\un $, and evolve them according to $C$. Pair the paths taken by the two chains. This is a pairing between paths in $\P_1$ and $\P_2$, because $C$ is a coupling. There is at least probability $p$ that the two chains end in states with matching values of $f$, so at least $p$ of the paths in $\P_1$ are paired with a path from $\P_2$ which ends at a state where $f$ takes the same value. Therefore $f(X_t)$ and $f(\un)$ overlap in at least $p$ of their area, as required.
\end{proof}

\begin{Lemma}
\label{lem:linearcombinationf}
Using the notation of Proposition \ref{prop:coupling2}, assume that for any two initial states $x_0$ and $y_0$, there is at least a probability $p$ that $f(X_t) = f(Y_t)$ when the two chains evolve according to the coupling $C$. Then if the two chains are started in arbitrary distributions $X_0$ and $Y_0$ rather than fixed states, there is still at least probability $p$ that $f(X_t) = f(Y_t)$.
\end{Lemma}
\begin{proof}
The proof is the same as that of Lemma \ref{lem:linearcombination}.
\end{proof}

It is not obvious that the bound on the mixing time given by Proposition \ref{prop:coupling2} decreases with $t$. Lemma \ref{lem:decreasing} will show that this is the case.

\begin{Remark}
\label{rem:stationarycontinue}
When constructing a coupling to be used with Proposition \ref{prop:coupling1}, it is possible to decree that if the two chains are in the same state, then they will move in the same way, guaranteeing that they will continue to agree with one another. When constructing a coupling for Proposition \ref{prop:coupling2}, this is still possible --- that is, two chains in the same state will continue to agree --- but the same cannot be done for values of $f$. It may be that two chains presently have the same value of $f$, but cannot be coupled so that after one step they have the same value of $f$. Example \ref{ex:cantcouplef} gives an example of this behaviour. This observation is the same as noticing that the values of $f$ may not form a quotient Markov chain of $\M$.
\end{Remark}

\begin{Example}
\label{ex:cantcouplef}

Consider the random-to-top walk, and let $f$ be the label of the second-to-top card of the deck. If two decks are currently in the states $x_0 = (1,2,3,4,5,\dots,n)$ and $y_0 = (3,2,1,4,5,\dots,n)$, then $f(x_0) = f(y_0) = 2$. However, the distributions of $f$ after one step, $f(x_1)$ and $f(y_1)$ are wildly different for the two chains, as shown in the following table.

\begin{center}
\begin{tabular}{ccc}
$a$ & $P(f(x_1) = a)$ & $P(f(y_1) = a)$ \\
\rule{0pt}{3ex}1 & $\frac{n-1}{n}$ & 0 \\
\rule{0pt}{3ex}2 & $\frac{1}{n}$ & $\frac{1}{n}$ \\
\rule{0pt}{3ex}3 & 0 & $\frac{n-1}{n}$ \\
\end{tabular}
\end{center}
\end{Example}

\begin{Remark}
\label{rem:exp}
When analysing the convergence of a Markov chain, it suffices to find the time at which the total variation distance from stationarity falls below, say, a quarter, because it then decays exponentially (for instance, see Section 4.5 of \cite{LPW}). The analogous result for the convergence of a statistic is false, as shown in Example \ref{ex:convergencetable}. It is still the case that the total variation distance of the distribution of the statistic from stationarity eventually falls off exponentially, just no longer that the speed of this decay is controlled by the time taken for the distance to fall below a quarter.
\end{Remark}

\begin{Example}
\label{ex:convergencetable}
Consider the following variation on the random-to-top walk on the permutations of a deck of $n$ cards. At each step choose a card uniformly at random, and move it to the top of the deck. The bottom card of the deck is stuck to the table, and attempts to move it only succeed with probability $\frac{1}{100}$, otherwise the order of the deck remains unchanged. Let $f$ be the label of the top card, and the original top and bottom cards be 1 and $n$.

After one step of the chain, the distribution of $f$ is

\begin{center}
\begin{tabular}{cc}
$a$ & $P(f(x) = a)$\\
\rule{0pt}{3ex}1 & $\frac{1.99}{n}$\\
\rule{0pt}{3ex}$2 \leq a \leq n-1$ & $\frac{1}{n}$\\
\rule{0pt}{3ex}$n$ & $\frac{0.01}{n}$\\
\end{tabular}
\end{center}

The total variation distance between the distribution of $f$ after just one step and the uniform distribution is less than $\frac{1}{n}$. After 50 steps, it is more likely than not that the bottom card has not moved, so the total variation distance between the distribution of $f(X_{50})$ and the uniform distribution is at least $\frac{1}{2n}$. 
\end{Example}

\begin{Remark}
\label{rem:ptpairs}
As a consequence of Remark \ref{rem:exp}, when using a coupling to examine convergence of a statistic, the relevant information is not just how long it takes for the total variation distance to drop below one quarter (or any other fixed small number), as it might be while using a coupling for the mixing time, but rather is a sequence of data points of the form ``for any starting point, after time $t_i$, the total variation distance of the statistic is less than $p_i$''.
\end{Remark}

\begin{Lemma}
\label{lem:decreasing}
Let $\M$ be a Markov chain, $f$ a function on the state space of $\M$, and $d(t)$ be the maximum distance (either total variation or separation) of $f$ from uniform after $t$ steps of $\M$, over all possible starting configurations. Then $d$ is a nonincreasing function of $t$.
\end{Lemma}
\begin{proof}
For any starting configuration, the distribution of $f$ after $t+1$ steps of $\M$ is a distribution of $f$ after $t$ steps of $\M$, starting from $t=1$. The definition of $d(t)$ is the maximum distance over all initial states, including this one, so the distance after $t+1$ steps is at most $d(t)$. Therefore $d(t+1) \leq d(t)$. 
\end{proof}

Unlike the convergence of actual Markov chains (Lemmas 3.7 and 4.5 of \cite{ADstrong}), in this setting the total variation distance from stationarity is not submultiplicative. The distance will be submultiplicative eventually, but not at all times. This should be understood as the distance sometimes being small earlier than expected due to factors which do not control the long-term rate of convergence. Example \ref{ex:convergencetable} illustrates this behaviour. 

To this end, it will be convenient to work with the coupling time. The definition of a coupling time is modified in the natural way to allow for couplings of statistics on a Markov chain.  

\begin{Definition}
\label{def:couplingtime2}
Let $\M$ be a Markov chain and $C$ be a coupling on $\M$. The \emph{coupling time} is the (random) time until the two copies of $\M$ are in the same state, or the time until they have matching values of $f$, depending on the aim of the coupling in question.
\end{Definition}

\section{Examples}
\label{sec:examples}

One way to apply Proposition \ref{prop:coupling2} is to consider a coupling that has been successful in obtaining a bound for the mixing time via Proposition \ref{prop:coupling1}, and check what it says about a function $f$ of interest. This section details what some well-known couplings say about various statistics on the respective state spaces.

Keep in mind that these are only upper bounds on the time taken for a statistic to mix --- some statistics may well mix faster than shown by this particular coupling. These examples are intended to show how Proposition \ref{prop:coupling2} may be applied to reduce a mixing time problem to a question regarding a coupling time. The goal is not to analyse these coupling times in detail, so most examples will not have detailed bounds on $\tme$ for each $\eps$, but rather will describe the process in question and give some idea of how long it takes. In many examples, Chebyshev's inequality will give good upper bounds on the time taken.

\subsection{Random-to-top shuffles}
\label{sec:rtt}

The random-to-top shuffle consists of choosing a random card at each step, and moving it to the top of the deck. A coupling for two copies of this process is to choose cards with the same label in each deck --- for instance, moving both copies of the ace of spades to the top of their respective decks, regardless of their prior positions. The time taken for two copies of this process to couple is the coupon collector time $n\log(n)$.

Consider what this coupling has to say about each of the following statistics on $S_n$. 

\begin{enumerate}
\item The top card of each chain is the same after a single step, and this continues to be true after any number of steps. Therefore this statistic is exactly uniformly distributed after one step.
\item The second-to-top cards match as soon as two different labels have been chosen, and this continues to be true after this point. With $\G(p)$ denoting a geometric distribution, the coupling time is $$T \stackrel{d}{=} \G(1) + \G(\frac{n-1}{n}).$$ The expected time is $1 + \frac{n}{n-1}$. The probability that $T > 2$ is $\frac{1}{n}$, so the second-to-top card is within $\frac{1}{n}$ of uniform after two steps, by Proposition \ref{prop:coupling2}. This corresponds to the probability that the same card is chosen twice, so the original top card is more likely to be in the second position than other cards. 

Likewise, the probability that $T > 3$ is $\frac{1}{n^2}$, so the second-to-top card is within $\frac{1}{n^2}$ of uniform after three steps.
\item The location of the card labelled $1$ is the same in each deck as long as that label has been chosen. This coupling time is $$T \stackrel{d}{=} \G(\frac{1}{n}).$$ The expected time until this happens is $n$ steps, and the variance is $n-1$.
\item The locations of the cards labelled $1$ and $2$ match in the two decks as long as both of those labels have been chosen. This coupling time is $$T \stackrel{d}{=} \G(\frac{1}{n}) + \G(\frac{2}{n}).$$The expected time until this event is $\frac{n}{2} + n$, and the variance is $\frac{n}{(n-1)^2}$. Note that this time is the sum of the two worst terms of the coupon collector problem, in contrast to matching the top two cards of the deck, which was the sum of the two best terms. This generalises to attempting to match the locations of $k$ fixed cards.
\end{enumerate}

These examples appear to behave quite differently. When the value of the top card is coupled, it may be that the two chains both have the $1$ at the top, or the $2$, or the $k$, for any $k$. In contrast, when the location of the $1$ is coupled, it is always at the top of the deck. This latter might seem disconcerting --- a result about the location of the $1$ mixing is proven by coupling two instances of the chain, but the coupling always happens in a certain position.

This issue is reconciled by recalling the definition of total variation distance. Upper bounds on the total variation distance between two distributions do not guarantee that they have some chance of agreeing in any possible value, just that there is a certain chance that they agree at some value(s). After a single random-to-top step, there is a $\frac{1}{n}$ chance that the location of the $1$ matches, so the distribution of this statistic overlaps with the uniform distribution in at least $\frac{1}{n}$ of their area. This is true --- both distributions have at least a $\frac{1}{n}$ chance that the $1$ is in the top position.

Likewise, after two random-to-top steps, there is a $\frac{2n-1}{n^2}$ chance that the location of the $1$ matches, so as previously, the distribution of the statistic after two steps overlaps with the uniform distribution in at least $\frac{2n-1}{n^2}$ of their area. Again this is true, because both distributions have at least a $\frac{n}{n^2}$ chance that the $1$ is in the top position and at least a $\frac{n-1}{n^2}$ chance that the $1$ is in the second position.

Perhaps, then, it should be surprising that the identity of the top card is equally likely to take any value when it is matched. Indeed, that this happens means that the coupling time is also a strong stationary time, and so can be used to obtain bounds in separation distance rather than total variation distance. This will be discussed further in \cite{GWsstfeatures}.

Continuing with a more detailed example,

\begin{enumerate}[resume]
\item The top $k$ cards match as long as $k$ different labels have been chosen. The coupling time is $$T \stackrel{d}{=} \G(1) + \G(\frac{n-1}{n}) + \cdots + \G(\frac{n-k+1}{n}).$$Let $T$ be the time taken for this to occur. The expected value is 
\begin{align}
\label{eq:couponfirstkex}
E(T) &= \frac{n}{n} + \frac{n}{n-1} + \dots + \frac{n}{n-k+1}\\
& \leq \int_{0}^k \frac{n}{n-x} dx \nonumber \\
& = n \int_{0}^k \frac{1}{n-x} dx \nonumber \\
& = n\left[-\log(n-x)\right]_0^k \nonumber \\
& = n\log \frac{n}{n-k} \nonumber \\
\end{align}
The variance is 
\begin{equation}
\label{eq:couponfirstkvar}
\Var(T) = 0 + \frac{n}{(n-1)^2} + \frac{2n}{(n-2)^2} + \dots + \frac{(k-1)n}{(n-k+1)^2}
\end{equation}
For instance, for a game of poker in which only the top 17 cards of a 52--card deck are to be used, it might be demanded that the distribution of the identities and order of the top 17 cards of the deck were within 0.01 of uniform (Of course, this is not a reasonable shuffling scheme for a real deck of cards).

In this instance, $n=52$ and $k=17$, so $E(T) \leq 20.6$ and $\Var(T) \leq 4.3$. By Chebyshev's inequality, there is at most a 0.01 chance that $T$ is more than $E(T) + 10\sqrt{\Var(T)} \approx 41$. Therefore for this purpose, 41 random-to-top moves suffice, compared to the approximately $52\log 52 \approx 205$ required to get the state of the entire deck just to within 0.25 of uniform, 
\item The sets of cards in each quarter of the deck match once the top three quarters of the decks match. The coupling time is $$T \stackrel{d}{=} \G(1) + \G(\frac{n-1}{n}) + \cdots + \G(\frac{\frac{n}{4}+1}{n}).$$
\item The sets of cards in positions congruent to each $i$ modulo $4$ are not guaranteed to match until the entire decks are in the same configuration.
\item The parity of the permutations are not guaranteed to match until the entire decks are in the same order. This is a terrible bound, indicating only that the coupling was unsuited to this statistic. See Remark \ref{rem:badparity} for a better one.
\item The identity of the card immediately above the card labelled by $1$ matches as long as $1$ has been chosen. (The possible values of this statistic are $2$ to $n$, as well as a special value corresponding to the $1$ being on the top of the deck. If instead the definition of `previous card' were to wrap around, with the card above the top card being the bottom card, then this example would behave quite like the next).
\end{enumerate}

The coupling times so far considered in this section have been sums of independent geometric random variables. This will not always be the case. The reason it happens in these examples is that in each of them, the coupling is attempting to make matches. It either creates a match or does not, and the probability of creating a match depends only on the number of matches presently existing. It is also important that in all of these examples, matches are never destroyed. This is why it was necessary to check not only that the statistic matched at a certain time, but also that this would continue to be true after additional steps.

That is, in such examples, the coupling time is the hitting time of a relatively simple Markov chain. For example, to match the top four cards of the deck, the coupling time \[T \stackrel{d}{=} \G(\frac{4}{n})+\G(\frac{3}{n})+\G(\frac{2}{n})+\G(\frac{1}{n})\] is a hitting time for the Markov chain shown in Figure \ref{fig:topfourhitting}

In the next example, thinking of a coupling time as a hitting time enables the analysis of a more complicated Markov chain, where matches may be destroyed.

\begin{figure}
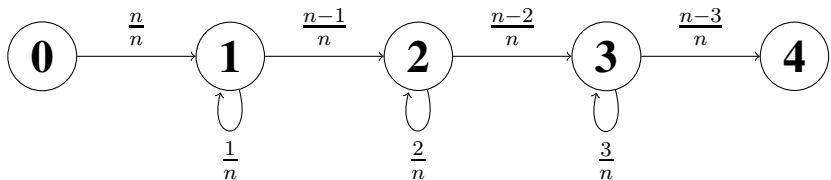

\figtopfourhitting
\caption[A Markov chain describing matching cards at the top of the deck]{This Markov chain describes the number of matching pairs of cards at the top of the deck when attempting to match the top four cards via random-to-top shuffles. The coupling time is the hitting time of the state $4$, starting from the state $0$.}
\label{fig:topfourhitting}
\end{figure}

Some more statistics:

\begin{enumerate}[resume]
\item The identities of the cards immediately below the $1$ match as long as the $1$ has been chosen, and if fewer than $n-1$ distinct labels have been chosen, the $1$ must have been chosen more recently than at least one of the other chosen labels (equivalently, the $1$ should not be on the bottom of the block of matching cards at the top of the deck). (As in the previous example, the possible values of this statistic are $2$ to $n$ and a special value corresponding to the $1$ being on the bottom of the deck). 

Notice that unlike the other statistics considered so far, it is possible that the coupling creates matches in this statistic and then breaks them again.

Unlike previous examples, this coupling time is not a sum of independent geometric random variables. To see why this is and how it may be analysed, consider running the coupling. The information needed to decide whether or not two copies of the chain have coupled is as follows
\begin{itemize}
\item How many distinct cards have been chosen
\item Whether or not the $1$ has been chosen
\item If the $1$ has been chosen, how many cards have been chosen and were last chosen before the last time the $1$ was chosen?
\end{itemize}

This information forms a quotient Markov chain, and understanding the behaviour of this chain suffices to understand the coupling time. Let $(k)$ denote the state where $k$ cards have been chosen, not including the $1$, and $(k,l)$ denote the state where $k$ cards have been chosen, including the $1$, and where $l$ of those cards were last chosen before the $1$ was. Equivalently, $l$ is the number of cards below the $1$ in the block of matching cards at the top of each deck. Figure \ref{fig:afteronehitting} illustrates this chain for $n=4$.

The goal is to understand the coupling time of the original chain. That is, after a certain number of steps, what is the probability that in each deck, the cards following the $1$'s are the same? If the quotient Markov chain is in state $(k,l)$ with $l > 0$ or $k = n$, then the two chains have coupled, so it suffices to understand the probability that after time $t$, the quotient chain is in such a state. When $n=4$, this is understanding the probability that if the Markov chain illustrated in Figure \ref{fig:afteronehitting} is started in the state $(0)$ then after $t$ steps it is at one of the blue states.

\begin{figure}
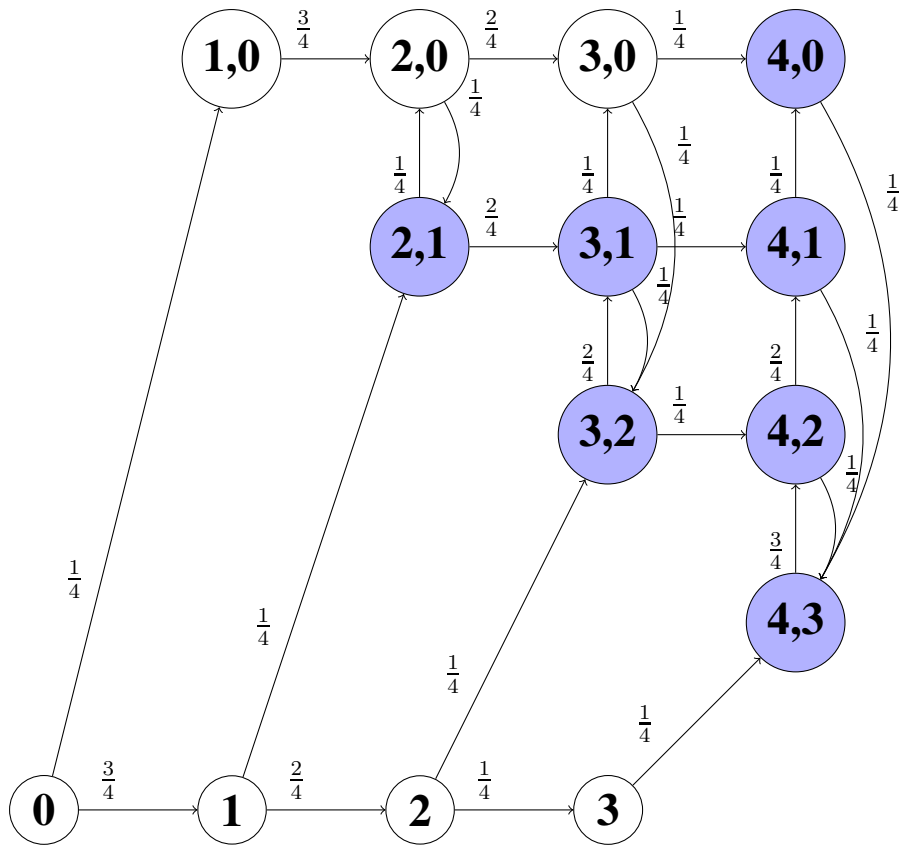

\figafteronehitting
\caption[A Markov chain recording necessary statistics for coupling]{For a deck of four cards, the Markov chain consisting of how many cards have been chosen, whether or not the $1$ has been chosen, and how many cards were last chosen before the $1$. Two copies of the random-to-top chain are coupled (for the identity of the card after the $1$) when this chain is in one of the blue states. Note that eventually this chain will settle into blue states forever, but before this time it is possible for it to enter a blue state and then leave again.}
\label{fig:afteronehitting}
\end{figure}

The following gives a sample bound on the coupling time when $n=52$. After $200$ steps, there is a $98\%$ chance that the $1$ has been chosen. As in Equations \ref{eq:couponfirstkex} and \ref{eq:couponfirstkvar}, after $200$ steps there is at least a 99\% chance that at least $42$ different cards have been chosen, so there is at least a $97\%$ chance that at least $42$ different cards have been chosen, including the $1$. In at most $\frac{1}{42}$ of paths leading to such outcomes, each other chosen card has been chosen after the $1$ last was. Thus with probability $\frac{41}{42}\cdot 0.97 \approx 95\%$, the cards immediately following the $1$ match after $200$ steps.

Two choices were made in this calculation --- the number of steps, but also to demand that at least $42$ distinct cards had been chosen. Changing these choices would produce slightly different bounds. 

\item The identities of the $k$ cards immediately after the card labelled by $1$ match as long as $1$ has been chosen, and if fewer than $n-1$ distinct labels have been chosen, $1$ must have been chosen more recently than at least $k$ of the other chosen labels (equivalently, $1$ should not be in the bottom $k$ cards of the block of matching cards at the top of the deck). This statistic takes values of ordered $k$--tuples, or smaller ordered tuples when the $1$ is close to the bottom of the deck.
\item The relative order of $1$ and $2$ matches as long as either label has been chosen. The coupling time is $$T \stackrel{d}{=} \G(\frac{2}{n}).$$
\item The relative order of $1$, $2$, ..., $k$ matches as long as all but one of these labels have been chosen. The coupling time is $$T \stackrel{d}{=} \G(\frac{k}{n}) + \G(\frac{k-1}{n}) + \cdots + \G(\frac{2}{n}).$$
\item The number of cards between $1$ and $2$ matches as long as both of these labels have been chosen. The coupling time is $$T \stackrel{d}{=} \G(\frac{2}{n}) + \G(\frac{1}{n}).$$ This is no better a bound than given for the stronger condition that the actual positions of cards $1$ and $2$ should match. It is unclear whether the weaker statistic mixes faster. That the bound is the same may be a weakness of the method, or it may be that a better coupling could be constructed.

\end{enumerate}
Some of these statistics are quotient chains of the random-to-top shuffle, and some are not, as follows:

These statistics are quotient chains. The identity of the top card, the identity and order of the top $k$ cards, the locations of any given set of cards, the parity of the permutation, and the relative order of a subset of cards.

These statistics do not form quotient chains:
\begin{itemize}
\item Identity of the second-to-top card. (Although this information is a subset of that contained in the identity and order of the top two cards, and that is a quotient chain)
\item Identity and order of the cards in any set $A$ of positions unless $A$ is a contiguous block at the top of the deck or $|A| = n-1$. (If the position $k$ is in $A$ but position $k-1$ is not, then if the card from position $k$ is moved to the top, it must be possible to deduce which card is now in position $k$ from only the information of which cards were in the positions in $A$. If besides $k-1$ there was another position not in $A$, those cards could be swapped without changing the available information, showing that this information is insufficient. If $A$ has size $n-1$ then the card in position $k-1$ is the only remaining card.)
\item The sets of cards in each quarter of the deck, either in blocks or interleaved.
\item The identities of the $k$ cards after a specific card.
\item Relative positions of a subset of cards. As commented above, the coupling does not treat this statistic any more specifically than just attempting to couple the exactly positions of those cards, and that is a quotient chain.
\end{itemize}

\begin{Remark}
\label{rem:badparity}
Notice that when $n$ is even, the parity of the permutation of the deck actually mixes perfectly in a single step, because exactly half of the moves correspond to multiplying by an odd permutation (and when $n$ is odd, it gets to within $\frac{1}{2n}$). However, using Proposition \ref{prop:coupling1} with the standard coupling of just choosing matching cards in each deck gives an upper bound of $n \log n$ steps for the mixing of this quantity, which is not at all good. This shows that the coupling used for the convergence of the chain need not be the best to use for the convergence of a statistic --- for permutation parity, for instance, there is a much better coupling which multiplies by either permutations of the same parity or of opposite parities, so that the resulting permutations are of the same parity. This gives that the parity mixes in one step.
\end{Remark}

In general, it is unclear whether upper bounds are bad because this shuffling technique is just not a good one for the statistic under consideration, or because the coupling was poorly chosen. For example, it is true that the random-to-top shuffle mixes the top card of the deck after a single step and takes many steps to mix the bottom card (at least $\frac{3}{4}n$ steps to get within $\frac{1}{4}$ of uniform, because it is impossible to get the original top card into the bottom quarter of the deck in fewer steps.), so some random walks are more suited to some statistics than to others. On the other hand, the previous example regarding permutation parity shows that a coupling may be ill-suited to a particular statistic, even if it gives a good bound for the convergence of the chain itself.

\subsection{Inverse riffle shuffles}
\label{sec:inverserifflestats}

As with the previous section, Proposition \ref{prop:coupling2} gives upper bounds on the mixing times of some statistics on $S_n$ under inverse riffle shuffles, which are modelled by assigning independent bits to each card and then sorting by those bits, breaking ties by the original order of the cards. Multiple steps of this process may be seen as assigning several bits, and sorting by the resulting base--$2$ string. 

Two instances of the inverse riffle shuffle may be coupled by, for each label, assigning that card either a $0$ or a $1$ independently with probability $\frac{1}{2}$, and making the same choice in each deck for the cards of that label. The two processes will agree when each pair of cards have been assigned different labels by at least one step.

It will be necessary to consider the strings assigned to the various cards. A subset of cards is considered to have distinct labels if each card in the subset has a different label, and unique labels if those labels are also not repeated among the remainder of the cards. 

\begin{Remark}
\label{rem:backwardsstrings}
These strings are growing right-to-left --- that is, least significant digit first. Taking a fixed number of steps of this chain, it will sometimes be convenient to consider the last step first, so that the most extreme changes in position are dealt with first, and the order is gradually refined with less and less impactful moves.
\end{Remark}

The coupling times in this section will not be sums of independent geometric random variables as they were in the previous section --- heuristically, there seemed to be something one-dimensional about most of the examples for the the random-to-top chain, where progress was only made in one direction, and it was possible to check how long it would take for each step, until the chains had coupled. Inverse riffle shuffles, on the other hand, change the positions of most of the cards at the same time. 

Analysis of coupling times using this coupling will require the treatment of an associated family of combinatorial problems regarding the strings assigned to the cards. Which cards have unique strings? Which positions contain cards with unique strings? The first results will be in answer to these questions, and these will be used to analyse some statistics on $S_n$.

\begin{Proposition}
\label{prop:matchcard}
For any card, the expected number of cards with the same string as this card after $t$ steps is $\frac{n-1}{2^t}$.
\end{Proposition}
\begin{proof}
The probability that two uniformly random binary strings of length $t$ are equal is $\frac{1}{2^t}$, and there are $n-1$ other cards.
\end{proof}

This result may be generalised:

\begin{Proposition}
\label{prop:matchcard2}
Let $A$ be a set of $N$ pairs of labels. After $t$ steps, the expected number of these pairs of cards which have the same string is $\frac{N}{2^t}$
\end{Proposition}
\begin{proof}
The proof is the same as that of Proposition \ref{prop:matchcard}.
\end{proof}

Perhaps surprisingly, the behaviour of the number of strings matching the card in a certain position (rather than the card of a certain value) behaves differently. This is because the position depends on the sorting, which depends on the assigned strings.

Exact calculations for this statistic are not included, but the following argument gives a heuristic for the scaling. 

\begin{Proposition}
\label{prop:matchposition}
For some fixed position $i$, let $A_t$ be the expected number of cards with the same string as the card in position $i$ after $t$ steps. Let $q$ be any real number greater than $\frac{1}{2}$. Then there is a constant $c$ depending on $q$ so that $A_t < cq^t(n-1)$.
\end{Proposition}
\begin{proof}
For this proposition, consider the most significant bit to be assigned first, as in Remark \ref{rem:backwardsstrings}.

This proposition is subtly different from the previous one, demonstrated with the following example. Consider a deck of four cards, and examine the number of cards with the same string as the card second from the top, whichever card this may be. Before any digits have been assigned, this is four. Now, when the first digit of each string is assigned, there are probabilities of 1, 4, 6, 4 and 1 sixteenths that 0,1,2,3 or 4 ones are assigned, respectively. This results in the number of cards sharing a string with the second card being 4,3,2,3 or 4, respectively --- this is different from the number of cards matching a specific card, which would be distributed binomially.

This differs from the behaviour of Proposition \ref{prop:matchcard} because the identity of the card in any given position depends on the assigned strings. Fortunately, the impact of this change is not particularly large, as will now be shown.

Assume that there are $k$ cards with strings matching the string assigned to the card in the $m$th position. After one step, these cards have been split binomially. The worst case (for there to be as many matches as possible) is for the $k$th position to be in the larger of the two blocks. This happens if it was central in the initial block of matching cards --- that is, the $k$ cards with the same string were in positions $m - \frac{k-1}{2}$ to $m + \frac{k-1}{2}$. 

So the number of cards with the same string as the card in any given position decays faster than the following process:
\begin{itemize}
\item Start with $n$
\item Repeatedly replace the current number $k$ with $k - r$, where $r$ is obtained by splitting $k$ into two binomial pieces and choosing the smaller.
\end{itemize}

This process iteratively replaces $k$ by $\frac{k}{2} + O(\sqrt{k})$, using Chebyshev's inequality on the binomial distribution. For large enough $k$, this decreases faster than replacing $k$ with $qk$. Let the constant $c$ be the difference for values of $k$ smaller than this. 
\end{proof}

It seems likely that the bound of Proposition \ref{prop:matchposition} could be improved to $A_t \leq c\frac{n-1}{2^t}$.

\begin{Example}
\label{ex:invrifflefeatures}
Consider the following statistics on $S_n$, to be studied via the inverse riffle shuffle. The point of these examples is that Proposition \ref{prop:coupling2} has reduced a mixing time problem to analysis of a coupling time. Estimates of the coupling times are given for some of the examples, but detailed analysis is not the goal of this section. Those examples involving the positions of certain cards or identities of cards in certain locations are known results --- see \cite{ADS} and \cite{repeatedcards1}, where these problems are analysed in greater detail.
\begin{enumerate}
\item The top card matches once the lexicographically first string is distinct from all others (equivalently, from the second). 

\begin{Proposition}
\label{prop:topcard}
The mixing time $\tme$ for the identity of the top card is at most $\log_2(n-1)-\log_2(\eps)$.
\end{Proposition}
\begin{proof}
If $m(t)$ is the number of strings equal to the first after $t$ steps, then $$\E{m(t+1)-1} < \frac{1}{2}(m(t)-1).$$ The quantity $m(0)$ is $n-1$, so after $\log_2(n-1)-\log_2(\eps)$ steps, there is at least a $1 - \eps$ chance that $m(t)$ is zero and so the top cards match. (This is better than would be given by Proposition \ref{prop:matchposition} because the first position is always at the top of its block, so is in the smaller piece exactly half the time)
\end{proof}
\item The second-to-top card matches once the second string (and hence also the first) is distinct from all others. 
\begin{Proposition}
The mixing time $\tme$ for the identity of the second card is at most $\log_2(n-1)+1-\log_2(\eps)$.
\end{Proposition}
\begin{proof}
Proposition \ref{prop:matchposition} suggests that this should take $\log_2(n) + c$ steps. This can be improved because the second position is near the top of the deck, so may only be in the smaller piece exactly once more than average.
\end{proof}
\item The $k$th-from-top card matches once the $k$th string is distinct from all others (equivalently, from the $(k-1)$th and $(k+1)$th). Proposition \ref{prop:matchposition} suggests that this takes $\log_2(n)+c$ steps.
\item The set of the top $k$ cards matches once the $k$th and $(k+1)$th strings are distinct. Again, Proposition \ref{prop:matchposition} suggests that this takes $\log_2(n)+c$ steps.
\item The identity and order of the top $k$ cards match once the top $(k+1)$ strings are all different. 
\begin{Proposition}
The mixing time $\tme$ for the identity and order of the top $k$ cards is $\log_2(n) + \log_2(k) - \log_2(\eps)$.
\end{Proposition}
\begin{proof}
After about $\log_2(\frac{n}{k})$ steps, there is a block of cards at the top of size slightly larger than $k$ with strings distinct from all others. Then by Proposition \ref{prop:matchcard2} with $A$ being the set of all pairs of those cards, after approximately another $2\log_2(k)+c$ steps, the probability that these cards will all have distinct strings is greater than $1-\frac{1}{2^c}$. So $\log_2(n) + \log_2(k) - \log_2(\eps)$ steps are enough.
\end{proof}
\item The location of the $1$ matches once the string assigned to that card is distinct from all others. 
\begin{Proposition}
\label{prop:locof1}
The mixing time $\tme$ for the location of the $1$ satisfies $$\tme \leq \log_2(n-1)-\log_2(\eps).$$
\end{Proposition}
\begin{proof}
Use Proposition \ref{prop:matchcard2} with $A$ being the set of pairs including $1$.
\end{proof}
\item The locations of $k$ specific cards match once the strings assigned to each are distinct from all others. 
\begin{Proposition}
\label{prop:locofk}
The mixing time $\tme$ for the locations of any $k$ specific cards is at most $\log_2(n) + \log_2(k)-\log_2(\eps)$.
\end{Proposition}
\begin{proof}
Use Proposition \ref{prop:matchcard2} with $A$ being the set of pairs including any of these cards. This gives that the time taken until the expected number of matches is below $\frac{1}{4}$ is at most $$\log_2(nk - \binom{k+1}{2}) -\log_2(\eps) \leq \log_2(n) + \log_2(k) -\log_2(\eps).$$
\end{proof}
\item The bridge hands in blocks match once the $(13a)$th and $(13a+1)$th strings are different for $a = 1,2$ and $3$. Proposition \ref{prop:matchposition} suggests that this takes about $\log_2(n)+O(1)$ steps.
\item The bridge hands distributed mod $4$ match once the entire deck matches.
\item The parity of the permutation matches once the entire deck matches. As before, this is an awful bound.
\item The card after the $1$ matches once both the string assigned to the $1$ and the next string are distinct from all others.
\item The relative order of the $1$ and $2$ match as soon as they are assigned different strings. \begin{Proposition} \label{prop:rel2}The mixing time $\tme$ of the relative order of the $1$ and the $2$ is $-\log_2(\eps)$\end{Proposition} \begin{proof}The relative order of the $1$ and $2$ matches once they are assigned different strings. This takes $\G(\frac{1}{2})$ steps.\end{proof}
\item The relative order of the $1$ through $k$ match as soon as they are all assigned different strings.\begin{Proposition} \label{prop:relk}The mixing time $\tme$ of the relative order of the $1$ through $k$ is $2\log_2(k)-\log_2(\eps)$\end{Proposition} \begin{proof}The relative order of these cards match once they are assigned different strings. Proposition \ref{prop:matchcard2} gives the result.\end{proof}
\end{enumerate}
\end{Example}

Each of these statistics regarding the inverse riffle shuffle process may be translated to the forwards riffle shuffle process. Typically, this interchanges the roles of card positions and card labels, as this is the difference between left-multiplication and right-multiplication in the symmtetric group. We do not give details of these translations here, but they may be found in Section 7.2.3 of \cite{GWThesis}.

\subsection{Random walk on the hypercube}
\label{sec:hypercube}

\begin{Example}
\label{ex:hypercubecoupling}
Consider the lazy nearest-neighbour walk on the hypercube described in Example \ref{ex:hypercube}. Given two instances of this walk in arbitrary initial states, they may be coupled as follows, as in \cite{aldous1983}:
\begin{Coupling}
\label{cou:hypercube}
At each step, choose a position $i$ and a random bit $x$, either $0$ or $1$. In each chain, change the value of the bit in position $i$ to $x$.
\end{Coupling}

The two chains will be in the same state once every position has been chosen at least once. The time taken until this happens is an instance of the coupon collector problem --- approximately $n\log(n)$ steps are required.
\end{Example}

Consider what this coupling says about some statistics on the hypercube using Proposition \ref{prop:coupling2}.

\begin{enumerate}
\item The value of the first bit (or the $k$th bit) matches once that bit is chosen, which takes on average $n$ steps.  
\item The number of `1's matches once every bit has been chosen. 
\item The position of the first `1' matches if there is some $k$ so that the first $k$ bits have all been chosen and for at least one of them, the last time it was chosen, it was set to 1.

Hence once the first $k$ bits have all been chosen, there is a probability of at least $(1-2^{-k})$ that the position of the first `1' matches. For example, if to find a time by which there is at least a $\frac{15}{16}$ chance that the position of the first `1' matches, consider the time taken until the first five (not four) bits have been chosen, which, by coupon collector theory, has expectation less than $\frac{7}{3}n$ and standard deviation less than about $\frac{4}{3}n$. 

So after time $t = \frac{7}{3}n + 6\cdot\frac{4}{3}n = \frac{31}{3}n$, Chebyshev's inequality says that there is at least a $\frac{35}{36}$ chance that the first five bits have all been chosen. There is a $\frac{31}{32}$ chance that they were set to something other than all zeros, so the chance that the first `1' matches after time $\frac{31}{3}n$ is at least $1 - \frac{1}{32} - \frac{1}{36} > \frac{15}{16}$. 

This statistic is another example of one where the coupling can create a match and then destroy it, as opposed to `nicer' statistics, where matches, once created, endure forever.
\end{enumerate}

\subsection{Random transpositions on \texorpdfstring{$S_n$}{the symmetric group}}

Consider the shuffling scheme on a deck of $n$ cards where at each step, two cards are chosen uniformly at random and interchanged. Choosing the same card twice is allowed, and in this case the order of the deck is left unchanged. Equivalently, this is the random walk on $S_n$ generated by the set of all $\binom{n}{2}$ transpositions, along with $n$ copies of the identity.

It will be more convenient to describe the moves slightly differently. Define $$a_{i,j} = \text{``swap the card with label $i$ with the card in position $j$''}.$$ The random transposition walk is equivalently described by choosing $i$ and $j$ uniformly between $1$ and $n$ and then applying $a_{i,j}$.

Two copies of this walk may be coupled, following \cite{AF}

\begin{Coupling}
\label{cou:randomtranspositions}
\hfill
\begin{itemize}
\item Choose $i$ and $j$ uniformly, $1 \leq i,j \leq n$.
\item In each chain, apply $a_{i,j}$
\end{itemize}
\end{Coupling}

To analyse this coupling, define a `match' to be a card which is in the same position in both decks. Observe that the number of matches does not decrease, and increases whenever neither the cards of label $i$ nor the cards in position $j$ presently match. According to this coupling, it takes approximately $n^2$ steps to couple the two chains. This shuffle actually mixes in $\f12n\log(n)$ steps (see \cite{randomtranspositions}), but no Markovian coupling can give this bound (consider two decks whose orders differ by a single transposition, and note that there's only a $\frac{2}{n^2}$ chance that they move to the same state, however they are coupled). See \cite{bormashenko2011coupling} for an amazing (non-Markovian) coupling for the random transposition walk, and a description of related problems.

It will soon be convenient to have some slight variants on this coupling. The previous coupling has the property that if the cards labelled by $k$ match in the two decks, then this match cannot be destroyed by choosing $i=k$, but can be by choosing $j=k$, in which case the match is replaced by the two cards labelled by $i$ matching instead. The analysis of some statistics will be easier if the coupling is edited so that matches are never destroyed. 

This does not represent any great change in what's going on --- there is a possibility that the cards labelled by $k$ match, and then this match is broken and replaced by the cards $i$ matching. This is counterbalanced by some other paths where a different pair of cards matches, but that match is broken and replaced by the cards $k$ matching.

To that end, here is a second coupling for this walk. 

\begin{Coupling}
\label{cou:randomtranspositionslabels}
Define $$a_{i,j} = \text{``swap the card with label $i$ with the card in position $j$''}.$$ and $$b_{i,j} = \text{``swap the card in position $i$ with the card in position $j$''}.$$
\begin{itemize}
\item Choose $i$ and $j$ uniformly, $1 \leq i,j \leq n$.
\item If the cards in position $j$ do not match, then apply $a_{i,j}$ in both chains.
\item If the cards in position $j$ did match, then instead apply $b_{i,j}$ in both chains.
\end{itemize}
\end{Coupling}

To see that this coupling restricts to the original random walk on both instances of the chain, observe that for any fixed $j$, $$\{a_{i,j}\}_{1\leq i\leq n} = \{b_{i,j}\}_{1\leq i\leq n}.$$ Because the decision as to whether to apply $a_{i,j}$ or $b_{i,j}$ depended only on the value of $j$, the coupling does restrict to the random transposition walk on both instances of the chain.

The analysis of this new coupling is exactly the same as the old --- the number of matches never decreases, and increases by one whenever neither the cards of label $i$ nor the cards in position $j$ currently match. However, it has the property that individual matches are never destroyed, while the previous coupling would destroy matches and replace them by others.

This modification ensured that once the cards labelled by $k$ matched, they would continue to match, albeit possibly in different positions. Alternatively, it could have been defined so that once there was a match in position $k$, there would continue to be a match in that position, although potentially of cards of a different value. 

To do this, here is a third coupling. 

\begin{Coupling}
\label{cou:randomtranspositionspositions}
Define $$a_{i,j} = \text{``swap the card with label $i$ with the card in position $j$''}.$$ and $$c_{i,j} = \text{``swap the card with label $i$ with the card with label $j$''}.$$
\begin{itemize}
\item Choose $i$ and $j$ uniformly, $1 \leq i,j \leq n$.
\item If the cards of value $i$ do not match, then apply $a_{i,j}$ in both chains.
\item If the cards of value $i$ did match, then instead apply $c_{i,j}$ in both chains.
\end{itemize}
\end{Coupling}

As in the previous case, to see that this coupling restricts to the original random walk on both instances of the chain, note that for any fixed $i$, $$\{a_{i,j}\}_{1\leq j\leq n} = \{c_{i,j}\}_{1\leq j\leq n},$$ and the decision as to whether to apply $a_{i,j}$ or $c_{i,j}$ depended only on the value of $i$.

The analysis of this coupling is the same as the others. The number of matches never decreases. Matches will stay in the same position, but may change in value.

It is also possible to make only part of this variation: 

\begin{Coupling}
\label{cou:randomtranspositionsposition1}
Define $$a_{i,j} = \text{``swap the card with label $i$ with the card in position $j$''}.$$ and $$c_{i,j} = \text{``swap the card with label $i$ with the card with label $j$''}.$$
\begin{itemize}
\item Choose $i$ and $j$ uniformly, $1 \leq i,j \leq n$.
\item If the cards of value $i$ match and are in position 1, then apply $c_{i,j}$ in both chains.
\item Otherwise apply $a_{i,j}$ in both chains.
\end{itemize}
\end{Coupling}

This coupling has the property that the number of matches never decreases, and that once there is a match in position $1$, there will always be a match in position $1$.

These couplings may be used to examine the convergence of some statistics. Appendix B of \cite{ADS} computes similar results for the mixing of the position of a single card and the positions of half of the cards, or equivalently, the card in a certain position and the values of the cards in a certain half of the positions. 

\begin{enumerate}
\item The top card. 
\begin{Proposition}
\label{prop:rtcardtop}
The mixing time $\tme$ for the top card is at most $T$, defined by $\Pr(\G(\frac{1}{n}) > T) \leq \eps$.
\end{Proposition}
\begin{proof}
Consider Coupling \ref{cou:randomtranspositionsposition1}. At each step, if the cards in position 1 do not match, there is a $\frac{1}{n}$ chance of this happening, by choosing $j=1$ and any $i$. Once the cards in position 1 do match, this will remain true, though the matching values may change. Hence the coupling time is $\G(\frac{1}{n})$. This completes the proof.
\end{proof}

Coupling \ref{cou:randomtranspositionsposition1} was used for this purpose, because Couplings \ref{cou:randomtranspositions} and \ref{cou:randomtranspositionslabels} do not preserve matches in position $1$, while Coupling \ref{cou:randomtranspositionspositions} will attempt to preserve matches in other positions, which can increase the time taken to create a match in position $1$.
\item The mixing time for the $k$th card is the same as that of the top card. If Coupling \ref{cou:randomtranspositionsposition1} is changed to preserve matches in position $k$ rather than position $1$, then this is exactly the same as the previous example. That is, it takes $\G(\frac{1}{n})$ steps.
\item The top two cards. 
\begin{Proposition}
The mixing time $\tme$ for the top two cards is at most $T$, defined by $\Pr(\G(\frac{1}{n})+\G(\frac{n-1}{n^2}) > T) \leq \eps$.
\end{Proposition}
\begin{proof}For this statistic, vary Coupling \ref{cou:randomtranspositionsposition1} to preserve matches in either of the top two positions. Then while there are no matches in positions $1$ or $2$, each step has a chance of $\frac{2}{n}$ to create one, by choosing $j=1$ or $j=2$, and any $i$. Once there is a match in either of these positions, each step has a chance of $\frac{1}{n}\frac{n-1}{n}$ of creating a match in the other position --- by choosing $j$ to be the other of $\{1,2\}$ and $i$ to be anything but the value involved in the existing match. Therefore the coupling time is $\G(\frac{1}{n})+\G(\frac{n-1}{n^2})$.\end{proof} 
\item Any two cards. As was the case for attempting to match the card in a single position, the previous argument did not rely on the positions chosen, so the time until there are matches in any two positions is the same. 
\item The cards in any $k$ positions. \begin{Proposition}
\label{prop:anykcards}
The coupling time for the cards in any $k$ positions to match is \[T = \G(\frac{k}{n})+\G(\frac{(k-1)(n-1)}{n^2})+\dots+\G(\frac{n-k+1}{n^2}).\]
\end{Proposition}
\begin{proof}
Use a variation of Coupling \ref{cou:randomtranspositionsposition1} which preserves matches in the relevant positions.
\end{proof}
For example, to match the cards in any $\frac{3n}{4}$ positions takes on average time 
\begin{align*}
& \frac{n}{\frac{3n}{4}} + \frac{n}{\frac{3n}{4}-1}\frac{n}{n-1} + \dots + n\frac{n}{n+1-\frac{3n}{4}} \\
&\leq 4n(1 + \frac{1}{2} + \frac{1}{3} + \dots + \frac{1}{\frac{3n}{4}}) \\
&\approx 4n\log(\frac{3n}{4}) \\
&\approx 4n\log(n)
\end{align*}
The interesting point here is not Proposition \ref{prop:anykcards} itself --- this bound is still larger than the true value of the mixing time of order of the entire deck, obtained by other means. Rather, it is that a coupling which gives a mixing time of the chain too large by a factor of $n$ can give the correct order of the mixing time of a fairly large portion of the deck.
\item The position of the card labelled by $k$. \begin{Proposition}
The mixing time $\tme$ for the location of the $1$ is at most $T$, defined by $\Pr(\G(\frac{1}{n}) > T) \leq \eps$.
\end{Proposition}
\begin{proof} In the same way that Coupling \ref{cou:randomtranspositionspositions} was modified to create Coupling \ref{cou:randomtranspositionsposition1}, Coupling \ref{cou:randomtranspositionslabels} may be modified to preserve matches only when the matching label is $1$. Then the proof is the same as that of Proposition \ref{prop:rtcardtop}, using positions rather than values. \end{proof}
\item The time for the positions of any $k$ cards to match is the same as in Proposition \ref{prop:anykcards}, but again working with positions rather than values.
\end{enumerate}

\subsection{Glauber dynamics for graph colourings}

Aldous and Fill in \cite{AF} present a coupling for a random walk on graph colourings. Consider a graph $G$ with $n$ vertices and maximal degree $r$, and a set of $c$ colours. A \emph{graph colouring} is an assignment of a colour to each vertex of the graph so that no two vertices of the same colour are connected by an edge.

In \cite{AF}, a coupling is used to show that if $c > 4r$ then the mixing time is bounded above by approximately $\frac{cn}{c-4r}\log(n)$. If the statistic of interest is the set of vertices of any given colour, then this coupling may be modified to show that this statistic mixes in approximately $\frac{cn}{c-3r}\log(n)$ steps. When $c$ is close to $4r$, this is significantly smaller. 

This is another situation where these results are interesting only in contrast to one another --- techniques other than coupling give better bounds.

\section{Further work}

We have discussed how coupling and strong stationary times may be used to give bounds for the convergence of statistics of a Markov chain. It would be useful to be able to go in the reverse direction.

\begin{Question}
Given bounds on the convergence of a suitably large collection of statistics on a Markov chain, is it possible to obtain bounds on the convergence of the Markov chain itself? 
\end{Question}

The examples of Section 7 of \cite{GWmutations} may be seen as examples where this is possible --- Propositions 30 and 31 of that section may be understood as making rigorous the heuristic that a deck of cards is mixed once each card is in a random position, which in those examples takes $n^3$ steps for any card, and a factor of $\log(n)$ because each card individually must have achieved this. 

Of course, care is needed here. Repeatedly applying powers of a single $n$--cycle will randomise the position of each card, but will certainly not result in a shuffled deck, because the positions of each card will be perfectly correlated with one another. A physical example of this is cutting a deck and placing the bottom portion on top. Regardless of how many and which cuts are made, the order of the deck is preserved up to cycling.

For some of the statistics considered in the present paper, a coupling immediately gave a good bound. For others, like the parity of the permutation, the bound was terrible, and a different argument was necessary.

\begin{Question}
Given a coupling or strong stationary time that gives good bounds for the convergence of a Markov chain, is it possible to predict for which statistics it will give good or bad bounds? How can better couplings or strong stationary times be designed for some statistics?
\end{Question}

\bibliographystyle{plain}
\bibliography{bib}
\end{document}